\documentclass[11pt]{amsart}
\usepackage[utf8]{inputenc}
\usepackage{amssymb, latexsym, amsmath, amsfonts, enumitem, fancyhdr}
\usepackage{amscd}
\usepackage{mathrsfs}

\usepackage{hyperref}
\usepackage{cleveref}

\newtheorem{thm}{Theorem}[section]

\newtheorem{cor}[thm]{Corollary}
\newtheorem{lem}[thm]{Lemma}
\newtheorem{prop}[thm]{Proposition}
\newtheorem{fact}[thm]{Fact}

\theoremstyle{definition}
\newtheorem{defn}[thm]{Definition}
\newtheorem{definition}[thm]{Definition}

\theoremstyle{remark}

\newtheorem{rem}[thm]{Remark}

\numberwithin{equation}{section}

\hypersetup{
  colorlinks,
  citecolor=black,
  linkcolor=black,
  urlcolor=blue
}

\theoremstyle{plain}

  {\end{enumerate}}

\newcommand{\C}{\mathbb{C}}

\newcommand{\T}{\mathbb{T}}

\hypersetup{
  colorlinks,
  citecolor=blue,
  linkcolor=blue,
  urlcolor=blue
}

\begin{document}

\title{Failure of Rational Dilation on the Tetrablock via Parrott Homomorphisms}
\author[Shubhankar Mandal]{Shubhankar Mandal}
\author[Samya Kumar Ray]{Samya Kumar Ray}
\date{\today}

\email{shubho0204@gmail.com}

\address[Samya Kumar Ray]{%
The Institute of Mathematical Sciences, 4th Cross Street, CIT Campus, Tharamani, Chennai, Tamil Nadu 600113, India
\and
Homi Bhabha National Institute, Training School Complex, Anushakti Nagar, Mumbai 400094, India%
}
\email{samya@imsc.res.in}


\pagestyle{headings}

\begin{abstract}
We determine the cotangent operator space structure at the origin of the tetrablock (denoted by $\mathbb E$). In particular we prove
that $\operatorname{COT}_0(\overline{\mathbb E})
  =\operatorname{MIN}(\ell_1^2\oplus_\infty\mathbb C)$
completely isometrically. This shows the existence of commuting tuple of operators having $\overline{\mathbb E}$ as a spectral set but not as a complete spectral set.
\end{abstract}

\maketitle
\begingroup
\def\del{\partial}
\let\phi\varphi

The notion of a spectral set is a fundamental tool in the study of operator tuples associated with various complex domains. The development of this direction can be traced back to J. von Neumann’s seminal result establishing that the closed unit disc is a spectral set for every contraction, \cite{vn} which is known as the celebrated von Neumann inequality. We briefly recall this relevant notion and terminology following Arveson’s framework \cite{wa}.

In order to define the notions of spectral set and complete spectral set in the multivariable setting, one requires an appropriate multivariable analogue of the spectrum. Let $\textbf{T}=(T_1,\dots,T_m)$ be an $m$-tuple of commuting bounded operators on a complex Hilbert space $\mathcal{H}$. In this context, L. Waelbroeck introduced the joint spectrum $\text{sp}(\textbf{T})$ of $\mathbf{T}$ in \cite{lw}, defined as the set of all $w:= (w_1,\dots,w_m)\in \mathbb{C}^m$ such that $p(w)$ belongs to the spectrum of $p(\textbf{T})$ for every polynomial $p$ in $m$-variables. The joint spectrum $\text{sp}(\textbf{T})$ is nonempty and compact. Moreover, if $p$ is a polynomial in $m$-variable with no zeros in $\text{sp}(\textbf{T})$, then $p(\textbf{T})$ is invertible. For further properties of the joint spectrum, we refer to \cite{wa}.

For $\Omega\subseteq\mathbb{C}^m$ a compact set, we define $$\operatorname{Rat}(\Omega)=\left\{ \frac{p}{q}:
p,q\in\mathbb{C}[z_1,\ldots,z_m],\;
q(w)\neq 0\ \text{for every }w\in\Omega
\right\},$$ where $\mathbb{C}[z_1,\ldots,z_m]$ denotes the set of all complex polynomials in $m$-variables. The functions in $\operatorname{Rat}(\Omega)$ form an algebra of holomorphic functions on $\Omega$. Let $\mathbf{T}$ be a commuting tuple of operators on a Hilbert space $\mathcal H$ with $\text{sp}(\mathbf{T})\subseteq\Omega.$ Since the polynomial $q$ is nonvanishing on $\text{sp}(\mathbf{T})$, $q(\mathbf{T})$ is invertible. Thus, for $f=p/q\in\operatorname{Rat}(\Omega)$, we may define
$f(\mathbf{T})=p(\mathbf{T})q(\mathbf{T})^{-1}.$
Moreover, the map $f\rightarrow f(\textbf{T})$ is clearly an algebra homomorphism from $\operatorname{Rat}(\Omega)$ into $\mathcal{B}(\mathcal H)$. A compact set $\Omega\subseteq \mathbb{C}^m$ is called a \textit{spectral set} for $\textbf{T}$ if $\Omega$ contains $\text{sp}(\textbf{T})$ and $$\|f(\textbf{T})\| \leq \sup\{|f(w)|: w \in \Omega\}=:\|f\|_{\Omega}$$ for every $f\in \operatorname{Rat}(\Omega)$.

 For each $k\geq1$, let $M_k(\operatorname{Rat}(\Omega))$ denote the algebra of all $k\times k$ matrices over $\operatorname{Rat}(\Omega)$. For $F=(f_{i,j})_{i,j=1}^k\in M_k(\operatorname{Rat}(\Omega))$ one defines $\|F\|_{M_k(\operatorname{Rat}(\Omega))} = \sup \{\|F(w)\|_{\text{op}} : w \in \Omega\}.$
This makes $M_k(\operatorname{Rat}(\Omega))$ again a normed algebra. Furthermore, the map $F\rightarrow F(\mathbf{T})$ is an algebra homomorphism. We say $\Omega$ is a \textit{complete spectral set} for $\mathbf{T}$ if $\text{sp}(\mathbf{T})\subseteq \Omega$ and $\|F(\textbf{T})\|_{op}\leq \|F\|_{M_k(\operatorname{Rat}(\Omega))}$ for all $F\in M_k(\operatorname{Rat}(\Omega))$ for all $k\geq 1.$

In a remarkable paper Arveson showed that $\Omega$ is a complete spectral set for $\mathbf{T}$ if and only if $\mathbf{T}$ admits a normal $\partial\Omega$-dilation; see \cite[Theorem 1.2.2 and its Corollary]{wa}, which we recall now. Let $\Omega$ be the spectral set for $\mathbf{T}$. We say that $\mathbf{T}$ has a \textit{rational dilation} or \textit{normal $\partial\Omega$-dilation} if there exist a Hilbert space $\mathcal{K}$, an isometry $V:\mathcal{H}\longrightarrow\mathcal{K}$ and a commuting $m$-tuple of normal operators $\mathbf{N}=(N_1,\ldots,N_m)$ on $\mathcal{K}$ with $\text{sp}(\mathbf{N})\subseteq \partial\Omega$ such that $$f(\mathbf{T})=V^*f(\mathbf{N})V,\qquad \text{for every } f\in\operatorname{Rat}(\Omega).$$
Here, $\partial\Omega$ is the \textit{Shilov boundary} of the algebra $\operatorname{Rat}(\Omega)$; see \cite[Chapter 9]{aw,ab} for the relevant definition of the Shilov boundary. It follows from the definition that if $\Omega$ is a complete
spectral set for $\mathbf{T}$, then $\Omega$ is also a spectral set
for $\mathbf{T}$. The rational dilation problem for a fixed compact
set $\Omega$ asks whether the converse holds for every commuting
operator tuple $\mathbf{T}$. By Arveson's theorem recalled above,
this is equivalent to asking whether every commuting tuple having
$\Omega$ as a spectral set admits a normal $\partial\Omega$-dilation.

One of the fundamental developments in operator theory is Sz.-Nagy’s unitary dilation theorem for contractions, \cite{bn}, which tells  rational dilation always holds on the closed unit disc of $\mathbb{C}$. Since then an important problem in operator theory has been to determine whether rational dilation holds on the closure of a bounded domain in $\mathbb{C}^m$. Note that the rational dilation holds for the bidisc $\overline{\mathbb{D}}^2$ (And\"{o} \cite{ta}), annulus (Agler \cite{ja}), and for the symmetrized bidisc (Agler-Young \cite{ay}). On the other hand, rational dilation fails, in general on planar domains with two or more holes (Agler-Harland-Raphael \cite{ahr} and Dritschel-McCullough \cite{dm}) and, on any norm unit ball in $\mathbb{C}^m$ for $m\geq3$ (Paulsen \cite{Paulsen1992} and Pisier \cite{Pisier2003}) following Parrott's example for failure of rational dilation on the closed tridisc $\overline{\mathbb{D}^3}$ in \cite{sp}. 

Let $\mathbb{E}$ denote the tetrablock 
\begin{equation}
	\mathbb{E}:=\{x=(x_1,x_2,x_3) \in \mathbb{C}^3:~1-x_1z-x_2w+x_3zw\neq 0~~\text{for all}~~z,w\in \overline{\mathbb{D}}\}.
\end{equation}
In 2007, Abouhajar,Young, and White \cite{AWY2007} introduced tetrablock domain as it arises naturally in connection with $\mu$-synthesis. 
The distinguished boundary of $\overline{\mathbb{E}}$ is denoted by $b\mathbb{E}$ which is also the Shilov boundary \cite[Section 7]{AWY2007}. In \cite{tb}, Bhattacharyya initiated the study of commuting operator tuples for which tetrablock is a spectral set. Such a tuple is referred to as a tetrablock contraction. For a tetrablock contraction, rational dilation is nothing but a tetrablock isometric dilation or a tetrablock unitary dilation (see \cite[Definition 5.3 and 5.5]{tb}). The authors in \cite{bs} constructed a tetrablock contraction that admits a tetrablock unitary dilation while failing to satisfy the proposed necessary condition in \cite{sp1}. This demonstrates that the condition is, in fact, not necessary. They also proposed several alternative necessary conditions for the existence of a tetrablock isometric lift; although they did not obtain an example of a tetrablock contraction that violates any of these conditions. Thus, the validity of rational dilation for the tetrablock remains an open problem. 

 In this article, we solve the rational dilation problem in negative for tetrablock. Our main result is the following (see Corollary \eqref{mosttrivia}).
\begin{thm}\label{mainthm}
There exists $n\in\mathbb N$ and matrices $A_1,A_2,A_3\in M_n$ such that
the commuting tuple
\[
 (T_1,T_2,T_3)\ \text{where}\ T_j=\begin{pmatrix}0&A_j\\0&0\end{pmatrix},
  \qquad j=1,2,3,
\]
has $\overline{\mathbb E}$ as a spectral set but not as a complete
spectral set. Consequently, rational dilation fails on the tetrablock.
\end{thm}

Our approach uses operator space techniques and Parrott-like homomorphisms.
These homomorphisms generalize Parrott's construction \cite{sp} and
have been studied by Misra and his collaborators in
\cite{MR1033916,MR1058968,MR1305512,Misra1994,BagchiMisra1995,BagchiBhattacharyyaMisra2002}.
Related questions concerning contractivity, complete contractivity,
and curvature inequalities are considered in
\cite{BhattacharyyaMisra2005,MisraPal2018}.
Also for further work on contractivity and operator-space structures on
finite-dimensional Banach spaces, see \cite{MR3979938,MR4171378}.

Our proof of Theorem \ref{mainthm} proceeds in three steps. To the best of our knowledge, this is the first time this method has been successfully tried outside unit ball.

\begin{itemize}
    \item[(i)]First, following the work of Misra
and Paulsen, the contractivity and complete contractivity of a
Parrott-like homomorphism on $\operatorname{Rat}(\Omega)$ are reduced
to estimates on directional derivatives at a point
$w\in\operatorname{Int}\Omega$; see
\cite[Lemma~3.3]{Misra1994},
\cite[Theorem~1.1]{BagchiMisra1995}, and
\cite[Theorem~5.4]{Paulsen1992}.
Remarkably, the closures of the collection of scalar derivatives and
their matrix-valued analogues are closed unit balls with respect to some norms that
together defines an operator-space structure.
Following Paulsen \cite[Proposition~3.2]{Paulsen1992}, this is the
cotangent operator space $\operatorname{COT}_w(\Omega)$. This step is well-known.
\item[(ii)] Second, we explicitly compute the matricial normed structure of cotangent operator space for the
tetrablock at the origin and prove the completely isometric
identification (see Theorem~\ref{cor:COT-MIN})
\[
  \operatorname{COT}_0(\overline{\mathbb E})
  =\operatorname{MIN}(\ell_1^2\oplus_\infty\mathbb C).
\]
The scalar computation uses extreme point methods. For the matrix-valued
computation, we construct rational interpolating functions and prove
their contractivity using the Redheffer product
\cite{Redheffer1960,MR1365226}. It is remarkable that $\operatorname{COT}=\operatorname{MIN}$ for a domain like tetrablock which is not a unit ball. For unit ball this fact is always true (see \cite{Paulsen1992}).

\item[(iii)] Finally, as $\operatorname{MIN}(Y)\neq\operatorname{MAX}(Y)$ for every
complex Banach space $Y$ of dimension $\geq 3$; see
\cite[Corollary~3.9]{Pisier2003}.
Taking $Y=\ell_1^2\oplus_\infty\mathbb C$ and applying Paulsen's
argument \cite[Theorem~5.4]{Paulsen1992} gives a Parrott-like
homomorphism on $\operatorname{Rat}(\overline{\mathbb E})$ that is
contractive but not completely contractive. The corresponding
finite-dimensional commuting tuple proves the main result.
This last implication is well-known, but we include short details
for the reader's convenience.
\end{itemize}
The paper is organized as follows. Section~\ref{se:prelimandnota}
recalls the required facts about extreme points, Redheffer products,
operator spaces, and cotangent matricial-norm structure.
In Section~\ref{sec:tetrablock-cotangent}, we compute the collection of scalar and
matrix-valued directional derivative and prove the above identification
in Theorem~\ref{cor:COT-MIN}.
Section~\ref{sec:parrott-rational-dilation} treats Parrott-like
homomorphisms, explains Paulsen's criterion in
Theorem~\ref{thm:cotangent-spectral-set}, and concludes the proof of
the main result in Corollary~\ref{mosttrivia}.

\endgroup

\section{Preliminaries and notations}\label{se:prelimandnota}
Unless mentioned otherwise all vector spaces and linear maps in our paper are over complex field. We denote by $M_{m,n}$ the vector space of all $m\times n$ complex
matrices, and write $M_n=M_{n,n}$. For $1\leq i\leq m$ and
$1\leq j\leq n$, we denote by $e_{ij}$ the matrix unit with $1$ in
the $(i,j)$-th position and zero elsewhere. Clearly,
$\{e_{ij}:1\leq i\leq m,\,1\leq j\leq n\}$ forms a basis of
$M_{m,n}$.

Let $X$ be a vector space. We denote by $M_{m,n}(X)$ the
set of all $m\times n$ arrays $(x_{ij})$ such that $x_{ij}\in X$
for all $1\leq i\leq m$ and $1\leq j\leq n$, and write
$M_n(X)=M_{n,n}(X)$. The map
\[
  \sum_{i=1}^{m}\sum_{j=1}^{n}e_{ij}\otimes x_{ij}
  \longmapsto (x_{ij})_{\substack{1\leq i\leq m\\1\leq j\leq n}}
\]
is a vector space isomorphism from $M_{m,n}\otimes X$ onto
$M_{m,n}(X)$, where $\otimes$ denotes the algebraic tensor product.

For two normed linear spaces $X$ and $Y$, we denote by
$\mathcal B(X,Y)$ the space of all bounded linear maps from $X$
to $Y$, and write $\mathcal B(X)=\mathcal B(X,X)$. For
$T\in\mathcal B(X,Y)$, its operator norm is denoted by
$\|T\|_{X\to Y}$.
We drop the subscript when the underlying spaces are clear.
We write $\ell_2^m$ for $\mathbb C^m$ equipped with its Euclidean
norm.

The $\ell_\infty$-direct sum of two Banach spaces $X$ and $Y$, denoted by $X\oplus_\infty Y$, is the algebraic direct sum $X\oplus Y$
equipped with the norm
\[
    \|(x,y)\|_{X\oplus_\infty Y}
    :=\max\{\|x\|_X,\|y\|_Y\}.
\]
For $1\leq p<\infty$, their $\ell_p$-direct sum, denoted by
$X\oplus_p Y$, is $X\oplus Y$ equipped with the norm
\[
    \|(x,y)\|_{X\oplus_p Y}
    :=\bigl(\|x\|_X^p+\|y\|_Y^p\bigr)^{1/p}.
\]
Each of these direct sums is a Banach space. We denote $\mathbb{C}^m$ with the $\ell_p$-norm by $\ell_p^m.$

\subsection{Convex sets and extreme points} Throughout, convexity in a complex vector space is understood with
respect to real convex combinations.

\begin{definition}\cite[Chapter 3]{Rudin1991}
Let $K$ be a convex subset of a complex vector
space $V$. A point $x\in K$ is called an \emph{extreme point} of $K$
if, whenever
\[
    x=ty+(1-t)z,
    \qquad y,z\in K,\quad 0<t<1,
\]
we have $y=z=x$. The set of extreme points of $K$ is denoted by
$\operatorname{ext}(K)$.
\end{definition}

The Krein--Milman theorem \cite[Theorem~3.23]{Rudin1991}
states that every nonempty compact convex
subset of a locally convex Hausdroff space is the closed convex hull
of its extreme points. The closed convex hull of a  set $S$ in a locally convex Hausdroff space is denoted by $\overline{\operatorname{conv}}(S).$

The following product formula is recorded for closed convex sets in
\cite[Exercise~II.10 and its solution, p.~41]{KilincKarzanNemirovski2025}.
\begin{fact}\label{fact:extreme-points-product}
Let $K_1$ and $K_2$ be nonempty convex subsets of finite-dimensional
complex vector spaces $V_1$ and $V_2$, respectively. Then
\[
    \operatorname{ext}(K_1\times K_2)
    =
    \operatorname{ext}(K_1)\times\operatorname{ext}(K_2).
\]
\end{fact}
\begin{fact}[{\cite[Chapter~V, \S7, Exercise~2, p.~144]{Conway1990}}]\label{fact:extreme-points-ell-one-two}
The extreme points of the closed unit ball of the complex space
$\ell_1^2$ are
\[
  \operatorname{ext}((\ell_1^2)_1)
  =\{(\alpha,0):\alpha\in\mathbb T\}
    \cup\{(0,\alpha):\alpha\in\mathbb T\}.
\]
\end{fact}
\subsection{Redheffer products}
Let $\mathcal X,\mathcal U,\mathcal Y,\mathcal Z$ and $\mathcal X_1,\mathcal U_1,\mathcal Y_1,\mathcal Z_1$ be Hilbert spaces. 
Suppose
\[
L=
\begin{pmatrix}
A & B\\
C & D
\end{pmatrix},
\qquad
L_1=
\begin{pmatrix}
A_1 & B_1\\
C_1 & D_1
\end{pmatrix},
\]
where
\[
L:\mathcal X\oplus\mathcal U\longrightarrow
\mathcal Y\oplus\mathcal Z,\ \text{and}\ L_1:\mathcal X_1\oplus\mathcal U_1\longrightarrow
\mathcal Y_1\oplus\mathcal Z_1.
\]
Assume moreover that $\mathcal U_1=\mathcal Z,$ $\mathcal X=\mathcal Y_1.$ We make the assumption
\begin{equation}\label{eq:redheffer-invertibility}
\tag{*}
I-B_1C \quad \text{is invertible},
\end{equation}
which implies that \(I-CB_1\) is also invertible. The Redheffer product was studied in \cite{Redheffer1960} (also see \cite[Chapter XIV]{FoiasFrazho1990}). We follow the exposition in \cite[Section 1]{MR1365226} and define Redheffer product by
\begin{equation}
\label{eq:redheffer-definition}
L\circ L_1
=
\begin{pmatrix}
A(I-B_1C)^{-1}A_1
&
B+A(I-B_1C)^{-1}B_1D
\\[2mm]
C_1+D_1C(I-B_1C)^{-1}A_1
&
D_1(I-CB_1)^{-1}D
\end{pmatrix}.
\end{equation}
\begin{prop}
\label{prop:redheffer-basic}
If \(L\) and \(L_1\) are contractions, isometries, coisometries, or
unitaries, then \(L\circ L_1\) is also a contraction, isometry, coisometry, or
unitary respectively.
\end{prop}
This is a special case of \cite[Lemma 1]{MR1365226}. Proposition~\ref{prop:redheffer-basic} is what we only need.

\subsection{Operator spaces}\label{sec:operator-spaces}
\begin{defn}[Abstract operator space]\label{def:abstract-operator-space}
An abstract operator space is a normed linear space $X$
endowed with a matricial norm structure, that is, a family of
normed linear spaces $\{(M_n(X),\|\cdot\|_n):n\geq 1\}$,
satisfying the following conditions:
\begin{itemize}
\item
The norm on $M_1(X)=X$ is the given norm of $X$.
\item
$\|A\oplus B\|_{n+m}=\max\{\|A\|_n,\|B\|_m\}$
for all $A\in M_n(X)$ and $B\in M_m(X)$, where $A\oplus B$
denotes the block diagonal array with diagonal blocks $A$ and $B$.
\item
$\|ACB\|_n
\leq
\|A\|_{\ell_2^m\to\ell_2^n}
\|C\|_m
\|B\|_{\ell_2^n\to\ell_2^m}$
for all $A\in M_{n,m}$, $B\in M_{m,n}$ and $C\in M_m(X)$.
\end{itemize}
\end{defn}

\begin{defn}[Concrete operator space]\label{def:concrete-operator-space}
A concrete operator space is a triple $(X,i,\mathcal H)$ where $X$
is a normed linear space, $\mathcal H$ is a Hilbert
space, and $i:X\to\mathcal B(\mathcal H)$ is a linear isometry.
\end{defn}

Let $(X,i,\mathcal H)$ be a concrete operator space. Note that
$\mathrm{id}_{M_{m,n}}\otimes i:
M_{m,n}(X)\to M_{m,n}(\mathcal B(\mathcal H))$ is a linear map
which acts by applying $i$ to each entry, where
$\mathrm{id}_{M_{m,n}}$ denotes the identity map on $M_{m,n}$.
Moreover, we can identify $M_{m,n}(\mathcal B(\mathcal H))$ with
$\mathcal B(\mathcal H^n,\mathcal H^m)$ as vector spaces, where
$\mathcal H^r$ denotes the Hilbert space direct sum of $r$ copies
of $\mathcal H$. Under this identification, an array
$(T_{ij})\in M_{m,n}(\mathcal B(\mathcal H))$ acts on
$\mathcal H^n$ by
\[
  (\xi_j)_{j=1}^n
  \longmapsto
  \left(\sum_{j=1}^n T_{ij}\xi_j\right)_{i=1}^m
  \in\mathcal H^m.
\]
Then, for $(x_{ij})\in M_{m,n}(X)$, define
\begin{equation}\label{eq:concrete-matrix-norm}
  \|(x_{ij})\|_{m,n}
  :=
  \bigl\|(i(x_{ij}))\bigr\|_{\mathcal H^n\to\mathcal H^m}.
\end{equation}
When $m=n$, we write $\|\cdot\|_n=\|\cdot\|_{n,n}$.
Clearly, the family $\{(M_n(X),\|\cdot\|_n):n\geq 1\}$
obtained from \eqref{eq:concrete-matrix-norm} makes $X$ an
abstract operator space, as the conditions in
Definition~\ref{def:abstract-operator-space} are all easy to check.
The following remarkable theorem of Ruan \cite{Ruan1988}
guarantees the converse.

\begin{thm}[Ruan's theorem]\label{thm:ruan}
Let $X$ be an abstract operator space with a matricial norm
structure $\{(M_n(X),\|\cdot\|_n):n\geq 1\}$. Then there exists
a complex Hilbert space $\mathcal H$ and a linear isometry
$i:X\to\mathcal B(\mathcal H)$ such that
\begin{equation}\label{eq:ruan-representation}
  \|(x_{ij})\|_n
  =
  \bigl\|(\mathrm{id}_{M_n}\otimes i)((x_{ij}))\bigr\|
       _{\mathcal H^n\to\mathcal H^n}
\end{equation}
for every $n\geq 1$ and every $(x_{ij})\in M_n(X)$.
\end{thm}
\begin{rem}
    In view of Theorem~\ref{thm:ruan}, we shall henceforth make no
distinction between abstract and concrete operator spaces and
simply use the term operator space.
\end{rem}
Note that $\mathcal{B}(\mathcal H)$ is an operator space with the canonical matrical norm structure from the identification $M_n(\mathcal{B}(\mathcal H))=\mathcal{B}({\mathcal H}^n)$. If $X$ and $Y$ are two operator spaces and $u:X\to Y$ is a linear map, we say $u$ is completely bounded (resp. complete contraction) if and only if $\sup\limits_{k\geq 1}\|u\otimes\text{id}_{M_k}\|_{M_k(X)\to M_k(Y)}<\infty$ (resp. $\leq 1$). We denote $u\otimes \text{id}_{M_k}$ by $u_k$ for all $k\geq 1.$

For a normed linear space $(X,\|\cdot\|_X)$, we denote its closed
unit ball by
\[
  (X,\|\cdot\|_X)_1:=\{x\in X:\|x\|_X\leq 1\}.
\]
When the norm is clear, we write $(X)_1$. We denote by $X^*$ the
space of all bounded linear functionals on $X$, equipped with its
usual dual norm.

Let $(X,\|\cdot\|_X)$ be a normed linear space. For every $n\geq 1$ and
$(x_{ij})\in M_n(X)$, we define
\begin{equation}\label{eq:min-matrix-norm}
  \|(x_{ij})\|_{M_n(\operatorname{MIN}(X))}
  :=
  \sup_{f\in(X^*)_1}
  \bigl\|(f(x_{ij}))\bigr\|_{\ell_2^n\to\ell_2^n}.
\end{equation} Then the family $\{(M_n(X),\|\cdot\|_{M_n(\operatorname{MIN}(X))}):n\geq 1\}$ is a matricial norm structure on $(X,\|\cdot\|_X),$ making it an operator space \cite{EffrosRuan2000,Pisier2003}. We denote $X$ endowed with this operator space
structure by $\operatorname{MIN}(X)$. 

The operator space $\operatorname{MAX}(X)$ is defined by equipping
$M_n(X)$, for each $n\geq 1$, with the norm
\begin{equation}\label{eq:max-matrix-norm}
  \|(x_{ij})\|_{M_n(\operatorname{MAX}(X))}
  :=
  \sup\left\{
    \bigl\|(T(x_{ij}))\bigr\|_{\mathcal H^n\to\mathcal H^n}:
    T:X\to\mathcal B(\mathcal H),\ \|T\|\leq 1
  \right\},
\end{equation}
where the supremum is taken over all complex Hilbert spaces
$\mathcal H$ and all contractive linear maps
$T:X\to\mathcal B(\mathcal H)$. These norms satisfy the conditions
in Definition~\ref{def:abstract-operator-space} \cite{EffrosRuan2000,Pisier2003}. We denote $X$ endowed with this operator space
structure by $\operatorname{MAX}(X)$. 

We refer to \cite{EffrosRuan2000,Pisier2003} for further
discussion of the operator spaces $\operatorname{MIN}(X)$ and
$\operatorname{MAX}(X)$.
\subsection{Matricial norms on derivatives}\label{subsec:matricial-derivatives}
Let $\Omega\subseteq\mathbb C^m$ be a compact subset, and let $\operatorname{Int}(\Omega)$ denote the interior of $\Omega$. Fix
$w\in\operatorname{Int}(\Omega)$. For
$f\in\operatorname{Rat}(\Omega)$, set
\[
  Df(w):=
  \left(
    \frac{\partial f}{\partial z_1}(w),\ldots,
    \frac{\partial f}{\partial z_m}(w)
  \right).
\]
We denote by $\mathcal D_{\Omega,w}\subseteq\mathbb C^m$ the
closure of the set
\[
  \left\{
    Df(w):
    f\in\operatorname{Rat}(\Omega),\
    \|f\|_\Omega\leq 1,\ f(w)=0
  \right\}.
\]
We identify $M_k\otimes\mathbb C^m$ with $M_k(\mathbb C^m)$ as mentioned in the beginning of Section \ref{se:prelimandnota}.
If $e_1,\ldots,e_m$ is the canonical basis of $\mathbb C^m$,
then every element of $M_k\otimes\mathbb C^m$ can be written
uniquely as $\sum_{r=1}^m B_r\otimes e_r$, where $B_r\in M_k$.
Thus, we also identify $M_k(\mathbb C^m)$ with the space of
$m$-tuples of $k\times k$ matrices and write
\[
  (B_1,\ldots,B_m)=\sum_{r=1}^m B_r\otimes e_r.
\]
Under these identifications, the $(i,j)$-th entry of the
corresponding array is $\sum_{r=1}^m(B_r)_{ij}e_r\in\mathbb C^m$.

For $F=(f_{ij})\in M_k(\operatorname{Rat}(\Omega))$, set
\[
  DF(w):=
  \left(
    \left(\frac{\partial f_{ij}}{\partial z_1}(w)\right)_{i,j=1}^k,
    \ldots,
    \left(\frac{\partial f_{ij}}{\partial z_m}(w)\right)_{i,j=1}^k
  \right).
\]
Thus, $DF(w)$ is an element of $M_k(\mathbb C^m)$, which may
also be written as $(Df_{ij}(w))_{i,j=1}^k$. We denote by
$\mathcal D_{\Omega,w}^{(k)}\subseteq M_k(\mathbb C^m)$ the
closure of the set
\[
  \left\{
    DF(w):
    F\in M_k(\operatorname{Rat}(\Omega)),\
    \|F\|_{M_k(\operatorname{Rat}(\Omega))}\leq 1,\ F(w)=0
  \right\},
\]
where $\|F\|_{M_k(\operatorname{Rat}(\Omega))}:=
  \sup_{z\in\Omega}\|F(z)\|_{\ell_2^k\to\ell_2^k}.$
In particular, $\mathcal D_{\Omega,w}^{(1)}=\mathcal D_{\Omega,w}$.
The following theorem is a paraphrased version of
\cite[Proposition~3.2]{Paulsen1992}.

\begin{thm}\label{thm:cotangent-matrix-norms}
For every $k\geq 1$, there exists a norm
$\|\cdot\|_{\mathcal D_{\Omega,w}^{(k)}}$ on
$M_k(\mathbb C^m)$ whose closed unit ball is
$\mathcal D_{\Omega,w}^{(k)}$. Thus, we have
\[
  \mathcal D_{\Omega,w}^{(k)}
  =
  \left(
    M_k(\mathbb C^m),
    \|\cdot\|_{\mathcal D_{\Omega,w}^{(k)}}
  \right)_1.
\]
Moreover, the family
$\{(M_k(\mathbb C^m),
\|\cdot\|_{\mathcal D_{\Omega,w}^{(k)}}):k\geq 1\}$
is a matricial norm structure making
$(\mathbb C^m,\|\cdot\|_{\mathcal D_{\Omega,w}^{(1)}})$
an abstract operator space. 
\end{thm}
Throughout we denote $\|\cdot\|_{\mathcal D_{\Omega,w}^{(1)}}$ by $\|\cdot\|_{\mathcal D_{\Omega,w}}.$
\begin{defn}[Cotangent operator space]\label{def:cotangent-operator-space}
The normed space $(\mathbb C^m,\|\cdot\|_{\mathcal D_{\Omega,w}})$,
equipped with the matricial norm structure given in
Theorem~\ref{thm:cotangent-matrix-norms}, is called the cotangent
operator space to $\Omega$ at $w$ and is denoted by
$\operatorname{COT}_w(\Omega)$.
\end{defn}
We refer to \cite{Paulsen1992} for more on this topic.
\section{Computation for operator space structure for \texorpdfstring{$\operatorname{COT}_0(\overline{\mathbb E})$}{COT at the origin of the closed tetrablock}}\label{sec:tetrablock-cotangent}
  Throughout this section we denote $\overline{\mathbb E}$ by $\Omega.$ The following rational functions play a crucial role in characterization of tetrablock \cite[Definition~2.1]{AWY2007}. For $z\in \mathbb{C}$ and $x= (x_1, x_2, x_3)\in \mathbb{C}^3$ we define
	\begin{align}
		\Psi(z,x)&= \frac{x_3z-x_1}{x_2z-1}\\
		\Upsilon(z,x)&= \Psi(z,x_2,x_1,x_3)=\frac{x_3z-x_2}{x_1z-1}.
	\end{align}
     By
\cite[Theorem~2.4, conditions~$(2)$ and~$(2')$]{AWY2007},
\begin{equation}\label{eq:tetrablock-test-function-bounds}
  |\Psi(z,x)|\leq1,
  \qquad |\Upsilon(z,x)|\leq1
  \qquad (x\in\Omega,\ z\in\mathbb D).
\end{equation}
Several algebraic characterizations of tetrablock are known; see \cite[Theorem 2.2 and Theorem 2.4]{AWY2007}. We use the following two facts about the tetrablock. First,
\begin{equation}\label{eq:tetrablock-coordinate-bounds}
  \mathbb E\subset\mathbb D^3,\qquad
  \Omega\subset\overline{\mathbb D}^{\,3};
\end{equation}
see the characterizations in
\cite[Theorems~2.2 and~2.4]{AWY2007}. Second,
\begin{equation}\label{eq:dist-boundary}
b\mathbb E
=
\{(\omega\overline{z},z,\omega):
z\in\overline{\mathbb D},\ \omega\in\mathbb T\},
\end{equation}
is the distinguished boundary \cite[Theorem~7.1]{AWY2007}. In particular, every scalar function continuous
on $\Omega$ and holomorphic on $\mathbb E$ has its supremum norm on $b\mathbb E$.

\subsection{Computing \texorpdfstring{$\mathcal D_{\Omega,0}$}{D at the origin}}\label{subsec:scalar-tetrablock-derivatives}

\begin{prop}\label{prop:scalar-tetrablock-derivatives}
For $\Omega=\overline{\mathbb E}$, we have
\begin{equation}\label{eq:scalar-tetrablock-derivative-ball}
  \mathcal D_{\Omega,0}
  =K:=\left\{(c_1,c_2,c_3)\in\mathbb C^3:
       |c_1|+|c_2|\leq1,\quad |c_3|\leq1\right\}.
\end{equation}

\end{prop}

\begin{proof}
We first prove $\mathcal D_{\Omega,0}\subseteq K$.
Let $f\in\operatorname{Rat}(\Omega)$ satisfy
$f(0)=0$ and $\|f\|_\Omega\leq1$, and write
$Df(0)=(c_1,c_2,c_3)$. For $u,v\in\mathbb T$, consider
\[
  h_{u,v}(\lambda)=(u\lambda,v\lambda,uv\lambda^2),
  \qquad \lambda\in\mathbb D.
\]
This map takes $\mathbb D$ into $\mathbb E$, since
\[
  1-u\lambda z-v\lambda w+uv\lambda^2zw
  =(1-u\lambda z)(1-v\lambda w)\neq0
\]
whenever $z,w\in\overline{\mathbb D}$. The usual Schwarz lemma
on $\mathbb D$, applied to $f\circ h_{u,v}$, gives
\[
  |uc_1+vc_2|=|(f\circ h_{u,v})'(0)|\leq1.
\]
Taking the supremum over $u,v\in\mathbb T$, we obtain
$|c_1|+|c_2|\leq1$. Similarly, the map
$h_3(\lambda)=(0,0,\lambda)$ takes $\mathbb D$ into
$\mathbb E$, because $|\lambda zw|<1$ for
$z,w\in\overline{\mathbb D}$. Applying again the Schwarz lemma
to $f\circ h_3$ yields $|c_3|\leq1$. This proves the desired inclusion.

We next show that
$\operatorname{ext}(K)\subseteq\mathcal D_{\Omega,0}$. 
Fix $0<r<1$, and define for $\alpha,\beta\in\mathbb T$
\begin{align*}
  \Psi_r(\alpha,x)
  &:=\frac{x_3r\alpha-x_1}{x_2r\alpha-1},\\
  \Upsilon_r(\beta,x)
  &:=\frac{x_3r\beta-x_2}{x_1r\beta-1}.
\end{align*}
Clearly for all $0<r<1$ and $\alpha,\beta\in\mathbb T$ we have 
$\Psi_r(\alpha,\cdot),\ \Upsilon_r(\beta,\cdot)\in \operatorname{Rat}(\Omega)$. Moreover, by
\eqref{eq:tetrablock-test-function-bounds}, we have that $\|\Psi_r(\alpha,\cdot)\|_{\Omega}\leq 1$ and $\|\Upsilon_r(\beta,\cdot)\|_{\Omega}\leq 1.$
Both $\Psi_r(\alpha,\cdot)$ and $\Upsilon_r(\beta,\cdot)$ vanish at $0$.
We compute
\[
  D\Psi_r(\alpha,\cdot)(0)=(1,0,-r\alpha),
  \qquad
  D\Upsilon_{r}(\beta,\cdot)(0)=(0,1,-r\beta).
\]
As
$\mathcal D_{\Omega,0}$ is a closed unit ball (Theorem~\ref{thm:cotangent-matrix-norms}) and also balanced we obtain from above
\begin{equation}\label{eq:tetrablock-extreme-gradients}
  (\beta,0,\alpha)\in \mathcal D_{\Omega,0} \ \text{and}\ (0,\alpha,\beta)\in \mathcal D_{\Omega,0},\ \text{for all}\ \alpha,\beta\in\mathbb T.
\end{equation}

Note that by Fact~\ref{fact:extreme-points-product} $\operatorname{ext}(K)=\operatorname{ext}((\ell_1^2)_1)\times \mathbb T.$ Hence from Fact~\ref{fact:extreme-points-ell-one-two} and \eqref{eq:tetrablock-extreme-gradients} we see that $\operatorname{ext}(K)\subseteq \mathcal D_{\Omega,0}$, implying $K\subseteq\mathcal D_{\Omega,0}$ as $\mathcal D_{\Omega,0}$ is  a closed convex set.
Together with the first inclusion, this proves
\eqref{eq:scalar-tetrablock-derivative-ball}.
\end{proof}
\begin{rem}\label{COTisdirect}
It follows from Proposition~\ref{prop:scalar-tetrablock-derivatives} that the underlying Banach space of
$\operatorname{COT}_0(\Omega)$ is
$\ell_1^2\oplus_\infty\mathbb C$ under the identification
$(c_1,c_2,c_3)\mapsto((c_1,c_2),c_3)$.
\end{rem}
\subsection{Computing \texorpdfstring{$\mathcal{D}_{\Omega,0}^{(k)}$}{D at matrix level k}:}
\begin{lem}\label{lem:pencil}
Let $B_1,B_2\in M_k$ satisfy $\sup\limits_{\lambda\in\mathbb T}\|B_1+\lambda B_2\|\leq1.$
Then we have
\begin{equation}\label{eq:pencil}
  \|aB_1+bB_2\|\leq\max\{|a|,|b|\}
  \qquad\text{for all}\ a,b\in\mathbb C.
\end{equation}
\end{lem}
\begin{proof}
If either $B_1$ or $B_2$ is zero, the conclusion is immediate.
Assume that both are nonzero. For $|a|=|b|=1$, taking
$\lambda=\overline a b$ gives
\[
  \|aB_1+bB_2\|
  =\|B_1+\overline a bB_2\|\leq1.
\]
Fix unit vectors $\xi,\eta\in\mathbb C^k$. The scalar polynomial
\[
  (a,b)\longmapsto\langle(aB_1+bB_2)\xi,\eta\rangle
\]
is therefore bounded in modulus by $1$ on $\mathbb T^2$.
Applying the maximum-modulus principle in each variable gives
the same bound on $\overline{\mathbb D}^{\,2}$. Taking the supremum
over $\xi,\eta$, we obtain $\|aB_1+bB_2\|\leq1$ whenever
$|a|,|b|\leq1$. Scaling by $\max\{|a|,|b|\}$ proves
\eqref{eq:pencil}; the case $a=b=0$ is immediate.
\end{proof}
\begin{lem}\label{lem:pencil-scalarization}
Let $B_1,B_2\in M_k$. Then
\[
\sup_{\lambda\in\mathbb T}\|B_1+\lambda B_2\|
=
\sup_{\|\xi\|_2=\|\eta\|_2=1}
\bigl(
|\langle B_1\xi,\eta\rangle|
+
|\langle B_2\xi,\eta\rangle|
\bigr).
\]
\end{lem}
\begin{proof}
For any $a,b\in\mathbb C$, we have
$|a|+|b|=\sup_{\lambda\in\mathbb T}|a+\lambda b|$.
Hence
\begin{align*}
\sup_{\lambda\in\mathbb T}\|B_1+\lambda B_2\|
&=\sup_{\|\xi\|_2=\|\eta\|_2=1}
  \sup_{\lambda\in\mathbb T}
  |\langle(B_1+\lambda B_2)\xi,\eta\rangle|\\
&=\sup_{\|\xi\|_2=\|\eta\|_2=1}
  \bigl(
  |\langle B_1\xi,\eta\rangle|
  +|\langle B_2\xi,\eta\rangle|
  \bigr).
\end{align*}
\end{proof}
\begin{lem}\label{lem:redheffer}
Suppose $P,Q\in M_k$, $0\le r\le1$, $\|P\|\le r$, and $\|Q\|<1$. Then
\begin{equation}\label{eq:redheffer-estimate}
\left\|P+(1-r^2)Q(I+P^*Q)^{-1}\right\|\le1.
\end{equation}
\end{lem}
\begin{proof}
Put $\delta=1-r^2$. Let us define
\begin{equation}\label{eq:redheffer-matrices}
L=\begin{pmatrix}\sqrt\delta\,I&P\\-P^*&\sqrt\delta\,I\end{pmatrix},
\qquad L_1=\begin{pmatrix}0&Q\\0&0\end{pmatrix}.
\end{equation}
They are contractions. Indeed,
\begin{equation}\label{eq:L-star-L}
L^*L=\begin{pmatrix}\delta I+PP^*&0\\0&\delta I+P^*P\end{pmatrix}\le I,
\qquad \|L_1\|=\|Q\|<1.
\end{equation}
Here in the first inequality in above we use $PP^*\le r^2I$ and $P^*P\le r^2I$. Note that $I+QP^*$ is invertible as $\|QP^*\|\leq r\|Q\|<1$ which satisfies the 
invertibility condition \eqref{eq:redheffer-invertibility}.

In \eqref{eq:redheffer-definition} the $(1,2)$ th position of the Redheffer product $L\circ L_1$ is
\begin{align}
&P+\sqrt\delta\,I\,(I+QP^*)^{-1}Q\sqrt\delta\,I\notag\\
&=P+\delta(I+QP^*)^{-1}Q\notag\\
&=P+(1-r^2) Q(I+P^*Q)^{-1}.
\end{align}
Therefore, as both $L$ and $L_1$ are contractions by Proposition~\ref{prop:redheffer-basic} we obtain $P+(1-r^2)Q(I+P^*Q)^{-1}$ is a contraction. This completes the proof.
\end{proof}
\begin{thm}\label{thm:main}
Let $k\ge1$ and let $B_1,B_2,B_3\in M_k$ satisfy
\begin{equation}\label{eq:data}
\sup_{\lambda\in\mathbb T}\|B_1+\lambda B_2\|\le1,\qquad\|B_3\|\le1.
\end{equation}
Define
\begin{equation}\label{eq:definitions}
\begin{aligned}
L(x)&=x_1B_1+x_2B_2,\\
S(x)&=x_2B_1^*+x_1B_2^*,\\
\Delta(x)&=x_3-x_1x_2.
\end{aligned}
\end{equation}
Then for all $0<\rho<1$, the function
\begin{equation}\label{eq:formula}
F_\rho(x)=L(x)+\rho\Delta(x)B_3\bigl(I+\rho S(x)B_3\bigr)^{-1}
\end{equation}
belongs to $M_k(\operatorname{Rat}(\Omega))$ and satisfies
\begin{equation}\label{eq:interpolation}
\|F_\rho\|_\Omega\le1,\qquad F_\rho(0)=0,\qquad
DF_\rho(0)=(B_1,B_2,\rho B_3).
\end{equation}
\end{thm}
\begin{proof}
\textbf{Checking $F_\rho\in\operatorname{Rat}(\Omega)$}:
    Fix $0<\rho<1$. By Lemma~\ref{lem:pencil}
and \eqref{eq:definitions},
\begin{equation}\label{eq:S-bound}
  \|S(x)\|
  =\|\overline{x_2}B_1+\overline{x_1}B_2\|
  \leq\max\{|x_1|,|x_2|\}\leq1
  \qquad\ \text{for all}\ x\in \Omega,
\end{equation}
where the last inequality follows from
\eqref{eq:tetrablock-coordinate-bounds}.
Thus, by \eqref{eq:data} and \eqref{eq:S-bound},
\[
  \|\rho S(x)B_3\|
  \leq\rho\|S(x)\|\|B_3\|\leq\rho<1
  \qquad\text{for all}\ x\in \Omega.
\]
It follows that $I+\rho S(x)B_3$ is
invertible on a neighbourhood of $\Omega$, and hence
\[
  R_\rho(x):=(I+\rho S(x)B_3)^{-1}
\]
belongs to $M_k(\operatorname{Rat}(\Omega))$.
Hence it follows that $F_\rho\in M_k(\operatorname{Rat}(\Omega))$.

\textbf{Computing derivative of $F_\rho$:} We next compute the value and the first derivatives of $F_\rho$
at the origin. By \eqref{eq:definitions},
\[
  L(0)=S(0)=\Delta(0)=0,\qquad R_\rho(0)=I.
\]
Hence $F_\rho(0)=0$. Writing \eqref{eq:formula} as $F_\rho(x)=L(x)+\rho\Delta(x)B_3R_\rho(x),$
the product rule gives, for $j=1,2,3$,
\[
  \partial_jF_\rho(x)
  =\partial_jL(x)
   +\rho\,\partial_j\Delta(x)B_3R_\rho(x)
   +\rho\Delta(x)B_3\partial_jR_\rho(x).
\]
Since $\Delta(0)=0$ and $R_\rho(0)=I$, we obtain
\[
  \partial_jF_\rho(0)
  =\partial_jL(0)+\rho\,\partial_j\Delta(0)B_3,\qquad\text{for all }1\leq j\leq 3.
\]
Now we can easily compute that
\[
  (\partial_1L,\partial_2L,\partial_3L)=(B_1,B_2,0),
\]
and
\[
  (\partial_1\Delta,\partial_2\Delta,\partial_3\Delta)=(-x_2,-x_1,1).
\]
Consequently, we have
\begin{equation}\label{eq:derivatives}
  DF_\rho(0)
  =\bigl(\partial_1F_\rho(0),\partial_2F_\rho(0),\partial_3F_\rho(0)\bigr)
  =(B_1,B_2,\rho B_3).
\end{equation}

\textbf{Checking  $\|F_\rho\|_{\Omega}\leq 1$:} Now we show that $\|F_\rho(x)\|\leq 1$ for all $x\in\Omega.$ Note that by considering the holomorphic maps $x\mapsto \langle F_\rho(x)\xi,\eta\rangle,$ for unit vectors $\xi,\eta\in\mathbb C^k$, it is enough to show that $\sup\limits_{x\in b\mathbb E}\|F_\rho(x)\|\leq 1.$ To this end take $x=(\omega\overline z,z,\omega)\in b\mathbb E$ as in \eqref{eq:dist-boundary}, and set
\begin{equation}\label{eq:boundary-data}
r=|z|,\qquad P=L(x)=\omega\overline zB_1+zB_2,\qquad Q=\rho\omega B_3.
\end{equation}
Equation \eqref{eq:pencil} gives $\|P\|\le r$, while $\|Q\|<1$. Moreover, we observe
\begin{align}
\omega P^*&=\omega(\overline\omega zB_1^*+\overline zB_2^*)
=zB_1^*+\omega\overline zB_2^*=S(x),\label{eq:S-on-boundary}\\
\Delta(x)&=\omega-(\omega\overline z)z=\omega(1-r^2).\label{eq:Delta-on-boundary}
\end{align}
Consequently, we obtain
\[
\rho S(x)B_3=P^*(\rho\omega B_3)=P^*Q,\qquad
\rho\Delta(x)B_3=(1-r^2)Q.
\]
Substitution into \eqref{eq:formula} yields
\begin{equation}\label{eq:F-on-boundary}
F_\rho(x)=P+(1-r^2)Q(I+P^*Q)^{-1}.
\end{equation}
The proof is completed by applying Lemma \ref{lem:redheffer}.
\end{proof}

\begin{cor}\label{cor:Omega}
For every $k\ge1$,
\begin{equation}\label{eq:Omega-ball}
\mathcal{D}_{\Omega,0}^{(k)}=\left\{(B_1,B_2,B_3)\in M_k(\C^3):
\sup_{\lambda\in\T}\|B_1+\lambda B_2\|\le1,\ \|B_3\|\le1\right\}.
\end{equation}
\end{cor}
\begin{proof}
Let $F\in M_k(\operatorname{Rat}(\Omega))$ be such that $F(0)=0,$ $\|F\|_{M_k(\operatorname{Rat}(\Omega))}\leq 1$ and $DF(0)=(B_1,B_2,B_3).$ For any $\xi,\eta\in \mathbb{C}^k$ with $\|\xi\|_2\leq 1$ and $\|\eta\|_2\leq 1,$ we have that $x\mapsto \langle F(x)\xi,\eta\rangle$ is in $\operatorname{Rat}(\Omega),$ fixing the origin and supremum norm less than or equal to $1.$ Therefore, by Proposition \ref{prop:scalar-tetrablock-derivatives} we get \[ |\langle B_1\xi,\eta\rangle|+|\langle B_2\xi,\eta\rangle|\leq 1 \qquad\text{and}\qquad |\langle B_3\xi,\eta\rangle|\leq 1,\qquad\text{for all}\ \xi,\eta\in (\mathbb C^k,\|.\|_2)_1.
\] Therefore, by Lemma \ref{lem:pencil-scalarization} we get one side of the desired inclusion.

Conversely, by Theorem~\ref{thm:main} we see that $(B_1,B_2,\rho B_3)\in \mathcal{D}_{\Omega,0}^{(k)}$
for every $0<\rho<1$. As $\mathcal{D}_{\Omega,0}^{(k)}$ is closed letting $\rho\to 1$ proves the reverse inclusion of $(B_1,B_2,B_3)$ in \eqref{eq:Omega-ball}. This completes the proof of the corollary.
\end{proof}

\begin{thm}\label{cor:COT-MIN}
We have a completely isometric identification
\begin{equation}\label{eq:COT-MIN}
\operatorname{COT}_0(\Omega)=\operatorname{MIN}(\ell_1^2\oplus_\infty\mathbb C).
\end{equation}
\end{thm}
\begin{proof}
By Remark~\ref{COTisdirect}, the underlying Banach space of
$\operatorname{COT}_0(\Omega)$ is
$\ell_1^2\oplus_\infty\mathbb C$.
Let $\sum_{j=1}^3B_j\otimes e_j\in
M_k\otimes(\mathbb C^3,\|\cdot\|_{\mathcal D_{\Omega,0}})$.
We identify $(\ell_1^2\oplus_\infty\mathbb C)^*$ with
$\ell_\infty^2\oplus_1\mathbb C$. By \eqref{eq:min-matrix-norm},
\begin{equation}\label{eq:MIN-dual-formula}
  \left\|\sum_{j=1}^3B_j\otimes e_j\right\|
       _{M_k(\operatorname{MIN}(\ell_1^2\oplus_\infty\mathbb C))}
  =\sup_{\max\{|v_1|,|v_2|\}+|v_3|\leq1}
    \left\|\sum_{j=1}^3v_jB_j\right\|.
\end{equation}
For unit vectors $\xi,\eta\in\mathbb C^k$, duality gives
\[
\begin{aligned}
&\sup_{\max\{|v_1|,|v_2|\}+|v_3|\leq1}
  \left|\left\langle
    \left(\sum_{j=1}^3v_jB_j\right)\xi,\eta
  \right\rangle\right|\\
&\qquad=\max\left\{
    |\langle B_1\xi,\eta\rangle|
    +|\langle B_2\xi,\eta\rangle|,
    |\langle B_3\xi,\eta\rangle|
  \right\}.
\end{aligned}
\]
Interchanging the suprema and using Lemma~\ref{lem:pencil-scalarization},
we obtain
\begin{align*}
&\sup_{\max\{|v_1|,|v_2|\}+|v_3|\leq1}
  \left\|\sum_{j=1}^3v_jB_j\right\|\\
&\quad=\sup_{\|\xi\|_2=\|\eta\|_2=1}
  \max\left\{
    |\langle B_1\xi,\eta\rangle|
    +|\langle B_2\xi,\eta\rangle|,
    |\langle B_3\xi,\eta\rangle|
  \right\}\\
&\quad\leq\max\left\{
    \sup_{\|\xi\|_2=\|\eta\|_2=1}
    \bigl(|\langle B_1\xi,\eta\rangle|
         +|\langle B_2\xi,\eta\rangle|\bigr),
    \|B_3\|
  \right\}\\
&\quad=\max\left\{
    \sup_{\lambda\in\mathbb T}\|B_1+\lambda B_2\|,
    \|B_3\|
  \right\}.
\end{align*}
Conversely, $(v_1,v_2,v_3)=(1,\lambda,0)$ and $(0,0,1)$ belong
to the dual unit ball for every $\lambda\in\mathbb T$.
Taking these vectors in \eqref{eq:MIN-dual-formula} gives the
reverse inequality. Therefore,
\begin{equation}\label{eq:MIN-tetrablock-norm}
  \left\|\sum_{j=1}^3B_j\otimes e_j\right\|
       _{M_k(\operatorname{MIN}(\ell_1^2\oplus_\infty\mathbb C))}
  =\max\left\{
    \sup_{\lambda\in\mathbb T}\|B_1+\lambda B_2\|,
    \|B_3\|
  \right\}.
\end{equation}
By Corollary~\ref{cor:Omega} and
Theorem~\ref{thm:cotangent-matrix-norms}, the right-hand side is
also the norm of $\sum_{j=1}^3B_j\otimes e_j$ in
$M_k(\operatorname{COT}_0(\Omega))$.
Since $k$ is arbitrary, this proves \eqref{eq:COT-MIN}.
\end{proof}
\section{Parrott homomorphism and solution to rational dilation on tetrablock}\label{sec:parrott-rational-dilation}

Let us recall Parrott homomorphisms on a general compact set. We follow the exposition close to \cite{BagchiMisra1995}.
Let $\Omega\subseteq\mathbb C^m$ be compact with non-empty interior,
let $\omega\in\operatorname{Int}\Omega$, and let $A_1,\dots,A_m\in M_n$.
Write $\mathbf A=(A_1,\dots,A_m)$ and define
\begin{equation}\label{eq:parrott-derivative-pairing}
  \langle Df(\omega),\mathbf A\rangle
  :=\sum_{i=1}^m
       \frac{\partial f}{\partial z_i}(\omega)A_i,
  \qquad f\in\operatorname{Rat}(\Omega).
\end{equation}
The map
$\phi_\omega(\mathbf A,\cdot):\operatorname{Rat}(\Omega)\to M_{2n}$
defined by
\begin{equation}\label{eq:parrott-homomorphism}
  \phi_\omega(\mathbf A,f)
  =
  \begin{pmatrix}
    f(\omega)I_n & \langle Df(\omega),\mathbf A\rangle\\
    0           & f(\omega)I_n
  \end{pmatrix}
\end{equation}
is a continuous unital algebra homomorphism. 
We call such a homomorphism a \emph{Parrott-like homomorphism}.
It is \emph{contractive} if
\begin{equation}\label{eq:parrott-contractive}
  \sup\left\{
    \|\phi_\omega(\mathbf A,f)\|_{\mathrm{op}}:
    f\in\operatorname{Rat}(\Omega),\ \|f\|_\Omega\leq1
  \right\}
  \leq1.
\end{equation}

We say that $\phi_\omega(\mathbf A,\cdot)$ is
\emph{completely contractive} if
\begin{equation}\label{eq:parrott-completely-contractive}
  \left\|
    \bigl(\phi_\omega(\mathbf A,f_{ij})\bigr)_{i,j=1}^k
  \right\|_{\mathrm{op}}
  \leq
  \|(f_{ij})_{i,j=1}^k\|_{M_k(\operatorname{Rat}(\Omega))}
\end{equation}
for every $k\geq1$ and every
$(f_{ij})\in M_k(\operatorname{Rat}(\Omega))$.

This homomorphism is related to the notions of a \emph{spectral set}
and \emph{complete spectral set} for the commuting tuple
\begin{equation}\label{NWA}
  N(\omega,\mathbf A)
  =
  \left(
    \begin{pmatrix}
      \omega_1I_n & A_1\\
      0          & \omega_1I_n
    \end{pmatrix},
    \ldots,
    \begin{pmatrix}
      \omega_mI_n & A_m\\
      0          & \omega_mI_n
    \end{pmatrix}
  \right).
\end{equation}
Indeed, writing $N_i=\omega_iI_{2n}+Q_i$, where $Q_i= \begin{pmatrix}
      0 & A_i\\
      0          & 0
    \end{pmatrix},$ we have $Q_iQ_j=0$
for all $1\leq i,j\leq m$, which implies that  $N(w,\mathbf A)$ is a commuting tuple.
For every polynomial $p$, the matrix $p(N(\omega,\mathbf A))$ is
block upper triangular with both diagonal blocks equal to
$p(\omega)I_n$. Hence
\[
  \sigma\bigl(p(N(\omega,\mathbf A))\bigr)=\{p(\omega)\}.
\]
Thus $\omega$ belongs to the joint spectrum, while applying its defining
condition to the coordinate functions shows that it contains no other
point. Consequently,
\begin{equation}\label{eq:parrott-joint-spectrum}
  \text{sp}\bigl(N(\omega,\mathbf A)\bigr)
  =\{\omega\}\subseteq\Omega.
\end{equation}
Moreover, the rational functional calculus gives
\begin{equation}\label{eq:parrott-rational-calculus}
  f\bigl(N(\omega,\mathbf A)\bigr)
  =\phi_\omega(\mathbf A,f),
  \qquad f\in\operatorname{Rat}(\Omega).
\end{equation}
It follows from \eqref{eq:parrott-joint-spectrum} and
\eqref{eq:parrott-rational-calculus} that $\Omega$ is a spectral set for
$N(\omega,\mathbf A)$ if and only if \eqref{eq:parrott-contractive}
holds, and a complete spectral set if and only if
\eqref{eq:parrott-completely-contractive} holds.

For $F\in M_k(\operatorname{Rat}(\Omega))$, define
\[
  \langle DF(\omega),\mathbf A\rangle
  :=\sum_{j=1}^m
       \frac{\partial F}{\partial z_j}(\omega)\otimes A_j.
\]

The following lemma reduces contractivity at each matrix level to
functions vanishing at $\omega$.

\begin{lem}[{cf.\ \cite[Theorem~1.1]{BagchiMisra1995}}]\label{lem:parrott-derivative-test}
Let $k\geq1$. The inequality
\[
  \left\|
    \bigl(\phi_\omega(\mathbf A,f_{ij})\bigr)_{i,j=1}^k
  \right\|_{\mathrm{op}}
  \leq\|F\|_{M_k(\operatorname{Rat}(\Omega))}
\]
holds for every $F=(f_{ij})\in M_k(\operatorname{Rat}(\Omega))$
if and only if
\[
  \sup\left\{
    \|\langle DF(\omega),\mathbf A\rangle\|_{\mathrm{op}}:
    F\in M_k(\operatorname{Rat}(\Omega)),\
    \|F\|_{M_k(\operatorname{Rat}(\Omega))}\leq1,\
    F(\omega)=0
  \right\}
  \leq1.
\]
\end{lem}

\begin{thm}[{Paulsen \cite[Theorem~5.4]{Paulsen1992}}]\label{thm:cotangent-spectral-set}
Let $\Omega\subseteq\mathbb C^m$ be compact, let
$\omega\in\operatorname{Int}\Omega$, and let $Y$ be the underlying
Banach space of $\operatorname{COT}_\omega(\Omega)$. If
\[
  \operatorname{COT}_\omega(\Omega)
  \neq\operatorname{MAX}(Y)
\]
under the canonical identification, then there are $n\geq1$ and
$A_1,\dots,A_m\in M_n$ such that the commuting tuple
$N(\omega,\mathbf A)$ has $\Omega$ as a spectral set but not as a
complete spectral set.
\end{thm}

\begin{proof}
 By the definition of $\operatorname{MAX}(Y)$ as in \eqref{eq:max-matrix-norm}, there
exists $n\geq1$ and a contractive linear map $u:\operatorname{COT}_\omega(\Omega)\to M_n$ which is
not completely contractive. Set $A_j=u(e_j)$, where
$e_1,\dots,e_m$ is the standard basis of $\mathbb C^m$.

By the definition of the matrix unit balls of $\operatorname{COT}_\omega(\Omega)$, for every $k\geq1$,
\[
  \|u_k\|
  =\sup\left\{
    \|\langle DF(\omega),\mathbf A\rangle\|_{\mathrm{op}}:
    F\in M_k(\operatorname{Rat}(\Omega)),\
    \|F\|_{M_k(\operatorname{Rat}(\Omega))}\leq1,\
    F(\omega)=0
  \right\}.
\]
Since $\|u\|\leq1$ but $\|u_k\|>1$ for some $k$, Lemma~\ref{lem:parrott-derivative-test}
shows that $\phi_\omega(\mathbf A,\cdot)$ is contractive but not
completely contractive. 
\end{proof}
\begin{cor}\label{mosttrivia}
   There is a finite-dimensional Hilbert space $\mathcal H$ and a commuting tuple $(T_1,T_2,T_3)$ on $\mathcal H$ with $\text{sp}(T)\subseteq\overline{\mathbb E}$ such that $\overline{\mathbb E}$ is a spectral set for $(T_1,T_2.T_3)$ but not a complete spectral set for $(T_1,T_2,T_3).$
\end{cor}
\begin{proof}
    Note that we have already proved in Theorem \ref{cor:COT-MIN} that $\operatorname{COT}_0(\overline{\mathbb E})=\operatorname{MIN}(Y)$ completely isometrically where $Y=\ell_1^2\oplus_\infty \mathbb C.$ By \cite[Corollary 3.9]{Pisier2003} we have $\operatorname{MIN}(\ell_1^2\oplus_\infty \mathbb C)$ is not completely isometric to $ \operatorname{MAX}(\ell_1^2\oplus_\infty \mathbb C).$ The proof of the corollary is then an application of Theorem \ref{thm:cotangent-spectral-set}.
\end{proof}
\textbf{Acknowledgements:} The second-named author acknowledges the support 
the Prime Minister Early Career Research Grant
(ANRF/ECRG/2024/000699/PMS), and the ANRF grant
ANRF/ARGM/2025/000895/MTR.
The first-named author acknowledges the hospitality of
the Institute of Mathematical Sciences, Chennai, where this work was carried out.

\textbf{AI Declaration:} The authors formulated the overall strategy, selected the methods, and developed, simplified, and modified the proofs. GPT-6 (OpenAI) assisted with mathematical discussions, parts of the proof development, and LaTeX editing. The authors carefully reviewed the suggestions arising from these AI interactions, revised and incorporated them where appropriate, and independently verified the resulting mathematical arguments. The authors take full responsibility for all mathematical content in the final manuscript.


\begin{thebibliography}{10}

\bibitem{AWY2007}
A.~A. Abouhajar, M.~C. White, and N.~J. Young, \emph{A {Schwarz} lemma for a
  domain related to {$\mu$}-synthesis}, J. Geom. Anal. \textbf{17} (2007), 717--750. Corrected version: \url{https://arxiv.org/abs/0708.0637v3}.

\bibitem{ja}
J.~Agler, \emph{Rational dilation on an annulus}, Ann. of Math.
  \textbf{121} (1985), 537--563.

\bibitem{ahr}
J.~Agler, J.~Harland, and B.~J. Raphael, \emph{Classical function theory,
  operator dilation theory, and machine computation on multiply-connected
  domains}, Mem. Amer. Math. Soc. 191(892) (2008), viii+159 pp.

\bibitem{ay}
J.~Agler and N.~J. Young, \emph{A commutant lifting theorem for a domain in
  {$\mathbb C^2$} and spectral interpolation}, J. Funct. Anal. \textbf{161}
  (1999), 452--477.

\bibitem{aw}
H.~Alexander and J.~Wermer, \emph{Several complex variables and {Banach}
  algebras}, third ed., Graduate Texts in Mathematics, vol.~35,
  Springer, 1998.

\bibitem{ta}
T.~And\^{o}, \emph{On a pair of commutative contractions}, Acta Sci. Math. 
  \textbf{24} (1963), 88--90.

\bibitem{wa}
W. Arveson, \emph{Subalgebras of {$C^*$}-algebras. {II}}, Acta Math.
  \textbf{128} (1972), 271--308.

\bibitem{BagchiBhattacharyyaMisra2002}
B.~Bagchi, T.~Bhattacharyya, and G.~Misra, \emph{Some thoughts on {Ando}'s
  theorem and {Parrott}'s example}, Linear Algebra Appl. \textbf{341} (2002),
  357--367.

\bibitem{BagchiMisra1995}
B.~Bagchi and G.~Misra, \emph{Contractive homomorphisms and tensor product
  norms}, Integral Equations Operator Theory \textbf{21} (1995),
  255--269.

\bibitem{bs}
J.~A. Ball and H.~Sau, \emph{Rational dilation of tetrablock contractions
  revisited}, J. Funct. Anal. \textbf{278} (2020), no.~1, 108275, 14 pp.

\bibitem{tb}
T.~Bhattacharyya, \emph{The tetrablock as a spectral set}, Indiana Univ. Math.
  J. \textbf{63} (2014), 1601--1629.

\bibitem{BhattacharyyaMisra2005}
T.~Bhattacharyya and G.~Misra, \emph{Contractive and completely contractive
  homomorphisms of planar algebras}, Illinois J. Math. \textbf{49} (2005),
  no.~4, 1181--1201.

\bibitem{ab}
A.~Browder, \emph{Introduction to function algebras}, W. A. Benjamin Inc., New York, 1969.

\bibitem{Conway1990}
J.~B. Conway, \emph{A course in functional analysis}, second ed., Graduate
  Texts in Mathematics, vol.~96, Springer-Verlag, New York, 1990.

\bibitem{dm}
M.~A. Dritschel and S.~McCullough, \emph{The failure of rational dilation on a
  triply connected domain}, J. Amer. Math. Soc. \textbf{18} (2005), 873--918.

\bibitem{EffrosRuan2000}
E.~G. Effros and Z.-J. Ruan, \emph{Operator spaces}, London Mathematical
  Society Monographs. New Series, vol.~23, The Clarendon Press, Oxford
  University Press, New York, 2000.

\bibitem{FoiasFrazho1990}
C.~Foia{\c{s}} and A.~E. Frazho, \emph{The commutant lifting approach to
  interpolation problems}, Operator Theory: Advances and Applications, vol.~44,
  Birkh{\"a}user Verlag, Basel, 1990.

\bibitem{KilincKarzanNemirovski2025}
F.~K{\i}l{\i}n{\c{c}}-Karzan and A.~Nemirovski, \emph{Essential mathematics for
  convex optimization}, Cambridge University Press, Cambridge, 2025,
  Accompanying \emph{Solutions to selected exercises}, Exercise~II.10 and its
  solution, p.~41, \url{https://www2.isye.gatech.edu/~nemirovs/KKN_SMM.pdf}.

\bibitem{Misra1994}
G.~Misra, \emph{Completely contractive {Hilbert} modules and {Parrott}'s
  example}, Acta Math. Hungar. \textbf{63} (1994), no.~3, 291--303.

\bibitem{MR1033916}
G.~Misra and N.~S. Narasimha~Sastry, \emph{Bounded modules, extremal problems,
  and a curvature inequality}, J. Funct. Anal. \textbf{88} (1990),
  118--134.

\bibitem{MR1058968}
G.~Misra and N.~S. Narasimha~Sastry, \emph{Completely bounded modules and associated extremal problems}, J.
  Funct. Anal. \textbf{91} (1990), 213--220.

\bibitem{MisraPal2018}
G.~Misra and A.~Pal, \emph{Curvature inequalities for operators in the
  {Cowen--Douglas} class and localization of the {Wallach} set}, J. Anal. Math.
  \textbf{136} (2018), 31--54.

\bibitem{MR3979938}
G.~Misra, A.~Pal, and C.~Varughese, \emph{Contractivity and complete
  contractivity for finite dimensional {Banach} spaces}, J. Operator Theory
  \textbf{82} (2019), no.~1, 23--47.

\bibitem{MR1305512}
G.~Misra and V.~Pati, \emph{Contractive and completely contractive modules,
  matricial tangent vectors and distance decreasing metrics}, J. Operator
  Theory \textbf{30} (1993), no.~2, 353--380.

\bibitem{sp1}
S.~Pal, \emph{The failure of rational dilation on the tetrablock}, J. Funct.
  Anal. \textbf{269} (2015), 1903--1924.

\bibitem{sp}
S.~Parrott, \emph{Unitary dilations for commuting contractions}, Pacific J.
  Math. \textbf{34} (1970), 481--490.

\bibitem{Paulsen1992}
V.~I. Paulsen, \emph{Representations of function algebras, abstract operator
  spaces, and {Banach} space geometry}, J. Funct. Anal. \textbf{109} (1992), 113--129.


\bibitem{Pisier2003}
G.~Pisier, \emph{Introduction to operator space theory}, London Mathematical
  Society Lecture Note Series, vol. 294, Cambridge University Press, Cambridge,
  2003.

\bibitem{MR4171378}
S.~K. Ray, \emph{On isometric embedding {$\ell_p^m\to S_\infty$} and unique
  operator space structure}, Bull. Lond. Math. Soc. \textbf{52} (2020), no.~3,
  437--447.

\bibitem{Redheffer1960}
R.~M. Redheffer, \emph{On a certain linear fractional transformation}, J. Math.
  Phys. \textbf{39} (1960), 269--286.

\bibitem{Ruan1988}
Z.-J. Ruan, \emph{Subspaces of {$C^*$}-algebras}, J. Funct. Anal. \textbf{76}
  (1988), no.~1, 217--230.

\bibitem{Rudin1991}
W.~Rudin, \emph{Functional analysis}, second ed., International Series in Pure
  and Applied Mathematics, McGraw-Hill, Inc., New York, 1991.

\bibitem{bn}
B.~Sz.-Nagy, \emph{Sur les contractions de l'espace de {Hilbert}}, Acta Sci.
  Math. (Szeged) \textbf{15} (1953), 87--92.

\bibitem{MR1365226}
D.~Timotin, \emph{Redheffer products and characteristic functions}, J. Math.
  Anal. Appl. \textbf{196} (1995), 823--840.

\bibitem{vn}
J.~von Neumann, \emph{Eine {Spektraltheorie} f{\"u}r allgemeine {Operatoren}
  eines unit{\"a}ren {Raumes}}, Math. Nachr. \textbf{4} (1951), 258--281.

\bibitem{lw}
L.~Waelbroeck, \emph{Le calcul symbolique dans les alg{\`e}bres commutatives},
  J. Math. Pures Appl. \textbf{33} (1954), 147--186.

\end{thebibliography}
\end{document}